\documentclass[reqno]{amsart}
\usepackage{amsmath,amssymb}

\theoremstyle{plain}
\newtheorem{theo}{Theorem}
\newtheorem*{theo*}{Theorem}

\newtheorem{lemm}[theo]{Lemma}

\theoremstyle{definition}

\numberwithin{equation}{section}
\numberwithin{theo}{section}

\newcommand{\Irr}{\mathrm{Irr}}
\newcommand{\Cl}{\mathrm{Cl}}
\newcommand{\II}{{I\!I}}
\newcommand{\m}{\mathfrak m}
\newcommand{\Gal}{{\rm Gal}}
\newcommand{\Tr}{{\rm Tr}}
\newcommand{\QQ}{\mathbb Q}

\newcommand{\EE}{\mathbb E}
\begin{document}
\title[Denseness results for zeros and roots of unity]{Denseness results for zeros and\\ roots of unity in character tables}
\author{Alexander Rossi Miller}
\address{Center for Communications Research, Princeton}
\begin{abstract}
For any irreducible character $\chi$ of a finite group $G$,
  let $\theta(\chi)$ denote the proportion of elements $g\in G$ for
  which $\chi(g)$ is either zero or a root of unity. Then 
  for any $L\in[\frac{1}{2},1]$ and any $\epsilon>0$, there exists
  an irreducible character $\chi$ of a finite group such that $|\theta(\chi)-L|<\epsilon$.
\end{abstract}
\maketitle
\thispagestyle{empty}
\section{Introduction}
For any irreducible character $\chi$ of a finite group $G$, let
\[\theta(\chi)=\frac{|\{g\in G: \chi(g)\text{ is zero or a root of unity}\}|}{|G|},\]
so $\theta(\chi)$ is the proportion of elements $g\in G$ for which $\chi(g)$ is zero or a root of unity.
Using a result of C.\ L.\ Siegel \cite{Siegel}, J.\ G.\ Thompson \cite[p.~46]{Isaacs} 
showed that $\theta(\chi)>\frac{1}{3}$.
We show that there are no gaps past $\frac{1}{2}$ in the following sense.

\begin{theo}\label{theorem 1}
  For any $L\in\left[\frac{1}{2},1\right]$ and any $\epsilon>0$, there exists
  an irreducible character~$\chi$ of a finite group such that $|\theta(\chi)-L|<\epsilon$. 
\end{theo}

We also establish results for the following statistics.
For any finite group~$G$, let $\Irr(G)$ denote 
the set of irreducible characters of $G$ and let $\Cl(G)$ denote the set of conjugacy classes of $G$. For $\chi\in\Irr(G)$, let 
\begin{alignat*}{2}
  z_I(\chi)&=\frac{|\{g\in G: \chi(g)=0\}|}{|G|},&
  z_\II(\chi)&=\frac{|\{g^G\in \Cl(G) : \chi(g)=0\}|}{|\Cl(G)|},\\
  u_I(\chi)&=\frac{|\{g\in G: |\chi(g)|=1\}|}{|G|},\qquad&
  u_\II(\chi)&=\frac{|\{g^G\in \Cl(G) : |\chi(g)|=1\}|}{|\Cl(G)|},
\end{alignat*}
and
 \[\theta_\II(\chi)
   =\frac{|\{g^G\in \Cl(G) : \chi(g)\text{ is zero or a root of unity}\}|}{|\Cl(G)|}.\]
Since $|\chi(g)|=1$ if and only if $\chi(g)$ is a root of unity, we have 
 $\theta=z_I+u_I$ and $\theta_\II=z_\II+u_\II$.
\begin{theo}\label{theorem 2}
  Let $f\in\{z_I,z_\II, u_I,u_\II,\theta_\II\}$.
  Then for any 
  $L\in[0,1]$ and any ${\epsilon>0}$, 
  there exists 
  an irreducible character $\chi$ of a finite group 
  such that $|f(\chi)-L|<\epsilon$.
\end{theo}

The parts of Theorem~\ref{theorem 2} concerning $z_I(\chi)$ and $z_\II(\chi)$ are local analogues of results established by the author in~\cite{Miller2019}.
For any finite group $G$, let 
\begin{align*}
  z_I(G)&={\EE}_{\chi\in\Irr(G)}(z_I(\chi))=
  \frac{|\{(\chi,g)\in \Irr(G)\times G : \chi(g)=0\}|}{|\Irr(G)\times G|}\\
\intertext{and}
z_\II(G)&={\EE}_{\chi\in\Irr(G)}(z_\II(\chi))=
  \frac{|\{(\chi,g^G)\in \Irr(G)\times \Cl(G) : \chi(g)=0\}|}{|\Irr(G)\times \Cl(G)|},
  \end{align*}
so $z_\II(G)$ is the fraction of the character table of $G$ that is covered by zeros.
The main result of \cite{Miller2019} is as follows.
\begin{theo*}
 The sets $\{z_I(G):|G|<\infty\}$ and $\{z_\II(G): |G|<\infty\}$
  are dense subsets of the interval $[0,1]$.
  \end{theo*}
For any finite group $G$, let
\begin{align*}
  \theta(G)&=\frac{|\{(\chi,g)\in \Irr(G)\times G : \chi(g)\text{ is zero or a root of unity}\}|}{|\Irr(G)\times G|},\\
  \theta_\II(G)&=\frac{|\{(\chi,g^G)\in \Irr(G)\times \Cl(G) : \chi(g)\text{ is zero or a root of unity}\}|}{|\Irr(G)\times \Cl(G)|},\\
  u_I(G)&=\frac{|\{(\chi,g)\in \Irr(G)\times G : \chi(g)\text{ is a root of unity}\}|}{|\Irr(G)\times G|},\\
  u_\II(G)&=\frac{|\{(\chi,g^G)\in \Irr(G)\times \Cl(G) : \chi(g)\text{ is a root of unity}\}|}{|\Irr(G)\times \Cl(G)|}.
\end{align*}
The global analogues of Theorems~\ref{theorem 1} and~\ref{theorem 2} are as follows.
\begin{theo}\label{global 1/2}
  For any $L\in\left[\frac{1}{2},1\right]$ and any $\epsilon>0$, there exists
  a finite group $G$ such that
  $|\theta(G)-L|<\epsilon$.
\end{theo}

\begin{theo}\label{global 01}
  Let $f\in\{z_I,z_\II, u_I,u_\II,\theta_\II\}$. Then for any
  $L\in[0,1]$ and any $\epsilon>0$, 
  there exists 
  a finite group $G$ such that $|f(G)-L|<\epsilon$.
\end{theo}

\section{Preliminaries}
Our constructions make use of extraspecial groups, dihedral groups, projective special linear groups, and the following important consequence of \cite[Theorem~8]{Miller2024}.
\begin{theo}\label{cav}
  If $\chi$ is an irreducible character of a finite nilpotent group $G$ and if $g\in G$, then $\chi(g)$ is a root of unity if and only if $\chi(1)=1$.
\end{theo}

For any algebraic integer $\alpha$ that belongs to a cyclotomic field $\QQ(\zeta)$, let
\[\m(\alpha)=\frac{1}{[\QQ(|\alpha|^2):\QQ]}\Tr_{\QQ(|\alpha|^2)/\QQ}(|\alpha|^2)=\frac{1}{|\Gal(\QQ(\zeta)/\QQ)|}\sum_{\sigma\in\Gal(\QQ(\zeta)/\QQ)}|\sigma(\alpha)|^2.\]
The following lemma and proof are well known. Cf.~\cite[p.\ 113]{Cassels}. 

\begin{lemm}\label{Alg Lemma}
  Let $\alpha$ be an algebraic integer that belongs to a cyclotomic field.
  Then the following are equivalent.
  \begin{enumerate}
  \item $\alpha$ is a root of unity.
  \item $|\alpha|=1$.
  \item $\m(\alpha)=1$.
  \end{enumerate}
\end{lemm}

\begin{proof}[Proof of Lemma~\ref{Alg Lemma}]%
Let $\zeta$ be a root of unity such that
$\alpha\in\QQ(\zeta)$. Let  $\mathcal G=\Gal(\QQ(\zeta)/\QQ)$.
Since $\mathcal G$ is abelian and contains the restriction
of complex conjugation, $|\alpha|=1$ if and only if
$|\sigma(\alpha)|^2=1$ for all $\sigma\in\mathcal G$. By 
a theorem of Kronecker, $|\sigma(\alpha)|^2=1$ for all $\sigma\in\mathcal G$
if and only if $\alpha$ is a root of unity.

For the equivalence of the second and third statements,  we may assume that $\alpha\neq 0$, since $\m(\alpha)=0$ if and only if $\alpha=0$.
Then by the inequality of the arithmetic and geometric means, we have 
\[\m(\alpha)\geq \left(
    \textstyle{\prod_{\sigma\in\mathcal G}}|\sigma(\alpha)|^2\right)^{1/|\mathcal G|}\geq 1,\] since the product 
$\prod_{\sigma\in\mathcal G}|\sigma(\alpha)|^2$ is a nonzero algebraic integer in
$\QQ(\zeta)$ that is invariant under $\mathcal G$ and is therefore a positive rational integer. Moreover,
equality holds in both places if and only if $|\sigma(\alpha)|^2=1$ for all $\sigma\in \mathcal G$, which again is equivalent to $|\alpha|=1$.
\end{proof}

\begin{proof}[Proof of Theorem~\ref{cav}]%
Let $\chi$ be an irreducible character of a finite nilpotent group $G$, and let $g\in G$.
Then by \cite[Theorem 8]{Miller2024}, 
\[\chi(g)=0\quad\text{or}\quad \mathfrak m(\chi(g))\geq 2^{|\{\text{primes
      dividing $\chi(1)$}\}|}.\]
Hence $\chi(g)$ is a root of unity if and only if $\chi(1)=1$.
\end{proof}

Our constructions in \S\ref{Proofs} will make use of direct products of certain groups. For the purposes of this paper, 
for any non-negative integer $n$ and  any irreducible character $\chi$ of a finite group $G$, let 
$\chi^n$ be shorthand for the irreducible character $\chi\times\chi\times\ldots\times \chi$
of the direct product $G^n=G\times G\times\ldots\times G$, where $\chi^0$ is the principal character of the trivial group~$G^0$.

\begin{lemm}\label{Prop Nilp}
  Let $u\in\{u_I,u_\II\}$.  If $G$ and $H$ are finite nilpotent groups and
  $\chi\times\psi\in\Irr(G\times H)$, then
$u(G\times H)=u(G)u(H)$ and $u(\chi\times \psi)=u(\chi)u(\psi)$.
\end{lemm}

\begin{proof}[Proof of Lemma~\ref{Prop Nilp}]%
  Let $g\in G$ and $h\in H$. By Theorem~\ref{cav},
  since $\chi\times \psi$ is an irreducible character of a finite nilpotent group,
  the value $\chi(g)\psi(h)$ is a root of unity if and only if the degree 
$\chi(1)\psi(1)$ is equal to $1$, which, by Theorem~\ref{cav}, is equivalent to
$\chi(g)$ and $\psi(h)$ being roots of unity.
\end{proof}

\begin{lemm}\label{Limit Lemma}
  Let $z\in\{z_I,z_\II\}$. Let $G$ be a finite nonabelian group and let
  ${\chi\in\Irr(G)}$ with $\chi(1)>1$. Then
  \begin{enumerate}
  \item $z(G^1),z(G^2),\ldots$ tends to $1$ with steps of size less than $z(G)$,
  \item $z(\chi^1),z(\chi^2),\ldots$ tends to $1$ with steps of size less than $z(\chi)$.
  \end{enumerate}
\end{lemm}

\begin{proof}[Proof of Lemma~\ref{Limit Lemma}]%
  Let $x\in\{G,\chi\}$. Then
  \begin{equation}\label{z step}
    z(x^{k+1})=z(x^k)+(1-z(x^k))z(x),\quad k=1,2,\ldots,
  \end{equation}
  so the sequence $z(x^k)$ ($k=1,2,\ldots$) is increasing, bounded, and thus
  convergent with limit $\ell$ satisfying $\ell=\ell+(1-\ell)z(x)$. By a theorem of Burnside,
  $z(x)\neq 0$, so $\ell=1$. By \eqref{z step}, $|z(x^{k+1})-z(x^k)|<z(x)$ for $k=1,2,\ldots.$
\end{proof}

For each prime power $q$, let
$\chi_q=\pi_q-1$, where 
$\pi_q$ is the permutation character
of the group $L_2(q)={\rm PSL}_2(q)$
acting faithfully in the natural way on the points of the projective line
over~$\mathbb F_q$.  
Recall that if $q$ is even, then 
$L_2(q)$ has exactly $q^3-q$ elements and $q+1$ conjugacy classes.

\begin{lemm}\label{L2q Lemma}
  Let $q$ be a prime power. Then
  $\chi_q$ is the unique irreducible character of $L_2(q)$ such that
  $\chi_q(1)=q$, and it has the following properties.
  \begin{enumerate}
  \item\label{L2qz}  $\chi_q$ vanishes on exactly $\frac{2}{\gcd(2,q)}$ conjugacy classes, each of size $\frac{|L_2(q)|}{q}$. 
  \item\label{L2qu} If $\chi_q(g)\neq 0$ and $g\neq 1$, then $\chi_q(g)=\pm1$.
  \end{enumerate}
\end{lemm}

\begin{proof}[Proof of Lemma~\ref{L2q Lemma}]%
  This follows from \cite[Chapter XI, Lemma~5.2 and Theorems 5.5--5.7]{HuppertBlackburn}
  and the character table of $L_2(q)$ \cite[pp.~402--403]{Jordan}.
\end{proof}

For each positive integer $n$, let $H_n$ denote a fixed extraspecial group
of order $2^{2n+1}$.
In general, for any positive integer $n$ and prime $p$, if
$E$ is an extraspecial group
of order $p^{2n+1}$, then
the center $Z(E)$ is cyclic of order $p$ 
and the irreducible characters of $E$ are well known.

\begin{lemm}\label{Extraspecial Characters}
  Let $E$ be an extraspecial group of order $p^{2n+1}$ for some prime $p$ and positive integer $n$.
  Then  $E$ has exactly $p^{2n}+p-1$ irreducible characters. 
  Exactly  $p-1$ irreducible characters of $E$ have degree greater than~$1$,
  and they are 
  $\frac{1}{p^n}\lambda^{E}$ for $\lambda\in\Irr(Z(E))\smallsetminus\{1\}$.
  In particular,
  for each $\chi\in\Irr(E)$ with $\chi(1)>1$, the zeros of $\chi$ are the elements $g\in E\smallsetminus Z(E)$. 
\end{lemm}

\begin{proof}[Proof of Lemma~\ref{Extraspecial Characters}]%
See \cite[Kapitel~V, Satz 16.14]{Huppert}.
\end{proof}

\begin{lemm}\label{Extraspecial Lemma}
If $E$ is an extraspecial group of order $p^{2n+1}$ for some prime $p$ and positive integer $n$, then
\begin{alignat}{2}
  u_I(E)&=\frac{p^{2n}}{p^{2n}+p-1},& z_I(E)&=\frac{(p-1)(p^{2n+1}-p)}{(p^{2n}+p-1)p^{2n+1}},\\
  u_\II(E)&=\frac{p^{2n}}{p^{2n}+p-1},&\qquad z_\II(E)&=\frac{(p-1)(p^{2n}-1)}{(p^{2n}+p-1)^2}.
\end{alignat}
\end{lemm}

\begin{proof}[Proof of Lemma~\ref{Extraspecial Lemma}]%
This follows from Lemma~\ref{Extraspecial Characters}. See also \cite{Miller2019}.
\end{proof}

For each positive integer $n$, let $G_n=\langle s,t\mid s^2=t^{2^n}=(st)^2=1\rangle$, so $G_n$ is the nilpotent dihedral group of order $2^{n+1}$ with
$2^{n-1}+3$ conjugacy classes. If $n=1$, then $G_n$ is the Klein 4-group.
If $n\geq 2$, then  $1$ and $t^{2^{n-1}}$ make up the center,
each noncentral element of $G_n$ lying in $\langle t\rangle$ is conjugate only to itself and its distinct inverse,
and the complement of $\langle t\rangle$ breaks into 2 classes of size $2^{n-1}$.
Recall
\cite{Huppert}
that $G_n$ has exactly 4 characters $\lambda_i$ of degree~1 and $2^{n-1}-1$ irreducible characters $\eta_1,\eta_2,\ldots,\eta_{2^{n-1}-1}$ of degree greater than $1$ given~by
\begin{equation}\label{Dihedral Characters}
  \eta_h(t^k)=2\cos\frac{\pi hk}{2^{n-1}},\qquad \eta_h(st^k)=0.
\end{equation}
Let $\nu_p(n)$ denote the $p$-adic valuation of an integer $n$.

\begin{lemm}\label{Dihedral Vanishing Lemma}
  Let $n$ and $h$ be positive integers such that 
  ${0<h<2^{n-1}}$, and 
  let $\eta_h$ be the irreducible character of $G_n$ given by \eqref{Dihedral Characters}. 
  \begin{enumerate}
  \item\label{D elt} $\eta_h(g)=0$ for exactly $2^{\nu_2(h)+1}+2^n$ elements $g\in G_n$.
  \item\label{D class} $\eta_h$ vanishes on exactly $2^{\nu_2(h)}+2$ conjugacy classes of $G_n$.
  \item\label{D First} The number of pairs $(\chi,g)\in\Irr(G_n)\times G_n$ that satisfy $\chi(g)=0$ is equal to $2^{n-1}(n-1)+2^{2n-1}-2^n$.
  \item\label{D Second} The number of zeros in the character table of $G_n$ is equal to \[2^{n-2}(n-1)+2^n-2.\] 
  \end{enumerate}
\end{lemm}

\begin{proof}[Proof of Lemma~\ref{Dihedral Vanishing Lemma}]%
  The integers $k$ that satisfy  
  $\eta_h(t^k)=\cos\frac{\pi hk}{2^{n-1}}=0$ and $0\leq k<2^n$ 
  are the solutions of 
  $2hk=2^{n-1}\bmod 2^n$ with $0\leq k<2^n$, which
  are
  \begin{equation}
    2^{n-\nu_2(h)-2}(1+2j),\quad 0\leq j< 2^{\nu_2(h)+1}.
  \end{equation}
  Hence $\eta_h$ has exactly $2^{\nu_2(h)+1}$ zeros in
  $\langle t\rangle$. Since 
  all elements of $s\langle t\rangle$
  are also zeros of $\eta_h$, we get a total of $2^{\nu_2(h)+1}+2^n$ zeros.
  For the second part, since $t^{2^{n-1}}$ is not a zero of $\eta_h$,
  each zero of $\eta_h$ that lies in $\langle t\rangle$
  belongs to a class of size~$2$, while the
  zeros of $\eta_h$ in $s\langle t\rangle$ split into two  classes.
  For the third part, the number of integers $0<\ell<2^{n-1}$ such
  that $\nu_2(\ell)=k$ ($0\leq k\leq n-2$) is $2^{n-k-2}$,
  so by the first part, the number of pairs $(\chi,g)\in \Irr(G_n)\times G_n$ satisfying $\chi(g)=0$ is 
  $\sum_{k=0}^{n-2} 2^{n-k-2}(2^{k+1}+2^n)$. 
  Lastly, by using the second part, the number of pairs
  $(\chi,g^{G_n})\in\Irr(G_n)\times \Cl(G_n)$ satisfying $\chi(g)=0$ is 
  $\sum_{k=0}^{n-2} 2^{n-k-2}(2^k+2)$.
\end{proof}

\begin{lemm}\label{Dihedral Lemma}
For any positive integer $n$, 
 \begin{alignat}{2}
 u_I(G_n)&=\frac{4}{2^{n-1}+3},& z_I(G_n)&=\frac{1}{2}\frac{2^n+n-3}{2^n+6},\\
 u_\II(G_n)&=\frac{4}{2^{n-1}+3},&\qquad z_\II(G_n)&=\frac{2^{n-2}(n+3)-2}{(2^{n-1}+3)^2}.
 \end{alignat}
\end{lemm}

\begin{proof}[Proof of Lemma~\ref{Dihedral Lemma}]%
This follows from Theorem~\ref{cav}, Lemma~\ref{Dihedral Vanishing Lemma}, and the fact that $G_n$ has exactly 4 characters of degree 1. 
\end{proof}

\section{Proofs of Theorems 1--4}\label{Proofs}
\begin{proof}[Proof of Theorem~\ref{theorem 1}]%
Let $\epsilon>0$.
Let $n\geq 2$ be an integer such that ${2^n>\frac{1}{\epsilon}}$.
Let $\gamma$ be the irreducible character $\eta_1$ of the dihedral group $G_n$ given by \eqref{Dihedral Characters} with $h=1$.
By Lemma~\ref{Dihedral Vanishing Lemma}\eqref{D elt}, $\gamma$ has 
exactly $2^n+2$ zeros, so 
\[z_I(\gamma)=\frac{1}{2}+\frac{1}{2^n}<\frac{1}{2}+\epsilon.\]
By Theorem~\ref{cav}, since $\gamma(1)>1$ and $G_n$ is nilpotent, $u_I(\gamma)=0$.

Let $q=2^r$ for some positive integer $r$ such that $q>\frac{1}{\epsilon}$. By Lemma~\ref{L2q Lemma}\eqref{L2qz}, $z_I(\chi_q)=\frac{1}{q}<\epsilon$.
By Lemma~\ref{L2q Lemma}\eqref{L2qu} and the fact that $u_I(\gamma)=0$, for each non-negative integer $k$, we have that 
$u_I(\gamma\times\chi_q^k)=0$ and hence 
\[\theta(\gamma\times\chi_q^{k+1})=z_I(\gamma\times\chi_q^{k+1})
  =z_I(\gamma\times \chi_q^k)+(1-z_I(\gamma\times \chi_q^k))z_I(\chi_q).\]
It follows that the sequence $\theta(\gamma\times \chi_q^k)$ is increasing, bounded, and therefore convergent with limit $\ell$ satisfying $\ell=\ell+(1-\ell)z_I(\chi_q)$, which implies $\ell=1$. 
Since $z_I(\gamma)<\frac{1}{2}+\epsilon$ and $z_I(\chi_q)<\epsilon$,
the sequence $\theta(\gamma\times\chi_q^k)$  ($k=0,1,2,\dots$) 
starts out less than $\frac{1}{2}+\epsilon$ and tends 
to $1$ with steps of size less than $\epsilon$. Hence, for any $L\in[\frac{1}{2},1]$, there exists a $k$ such that 
$|\theta(\gamma\times\chi_q^k)-L|<\epsilon$. 
\end{proof}

\begin{proof}[Proof of Theorem~\ref{global 1/2}]%
Let $\epsilon>0$.
By Lemma~\ref{Dihedral Lemma}, the nilpotent dihedral groups $G_n$ satisfy
\[u_I(G_n)=\frac{4}{2^{n-1}+3},\qquad z_I(G_n)=\frac{1}{2}\frac{2^n+n-3}{2^n+6}.\]
Let $l$ be a positive integer such that
\begin{equation}\label{GN start}
  u_I(G_l)<\frac{\epsilon}{2},\qquad \frac{1}{2}< z_I(G_l)<\frac{1}{2}+\frac{\epsilon}{2}.\end{equation}
By Lemma~\ref{Extraspecial Lemma}, the extraspecial groups $H_n$ of order $2^{2n+1}$ satisfy 
\[u_I(H_n)=\frac{2^{2n}}{2^{2n}+1} ,\qquad z_I(H_n)=\frac{2^{2n+1}-2}{2^{4n+1}+2^{2n+1}}.\]
Let $m$ be a positive integer such that $z_I(H_m)<\epsilon$.

By Lemma~\ref{Prop Nilp}, for $k=0,1,2,\ldots,$
we have $u_I(G_l\times H_m^k)=u_I(G_l)u_I(H_m)^k$ 
and therefore 
\begin{equation}\label{Master}
  \theta(G_l\times H_m^k)
  =u_I(G_l)u_I(H_m)^k+z_I(G_l)+(1-z_I(G_l))z_I(H_m^k).
\end{equation}
By Lemma~\ref{Limit Lemma}, the sequence
$z_I(H_m^k)$ ($k=0,1,2,\ldots$) tends to $1$ with
steps of size less than $\epsilon$.
Since $0<u_I(H_m)<1$, the sequence $u_I(H_m)^k$ tends to~$0$.
It follows by \eqref{GN start} and \eqref{Master} that the sequence $\theta(G_l\times H_m^k)$ ($k=0,1,2,\ldots$) starts out less than $\frac{1}{2}+\frac{\epsilon}{2}$  and tends to $1$
with steps of size less than $\epsilon$.
Hence, for any $L\in[\frac{1}{2},1]$, there exists a $k$ such that 
${|\theta(G_l\times H_m^k)-L|<\epsilon}$.
\end{proof}

\begin{proof}[Proof of Theorem~\ref{theorem 2}]%
Let $\epsilon>0$.
Let $z\in\{z_I,z_\II\}$.
Put $q=2^r$ for some positive integer $r$ such that $q>\frac{1}{\epsilon}$. 
By Lemma~\ref{L2q Lemma}\eqref{L2qz}, we have $z_I(\chi_q)=\frac{1}{q}$ and $z_\II(\chi_q)=\frac{1}{q+1}$, so $z(\chi_q)<\epsilon$. 
Therefore, by Lemma~\ref{Limit Lemma}, the
sequence $z(\chi^n_q)$ ($n=1,2,\ldots$) starts out less than $\epsilon$ and tends
to $1$ with steps of size
less than $\epsilon$. Hence, for any $L\in [0,1]$, there exists an $n$ such that $|z(\chi^n_q)-L|<\epsilon$.

Let $u\in\{u_I,u_\II\}$.
Put $q=2^s$ for some positive integer $s$ such that $q>\frac{2}{\epsilon}$. 
By Lemma~\ref{L2q Lemma}, for each positive integer $n$, 
$u(\chi_q^n)=u(\chi_q)^n=(1-\Delta_q)^n$, 
where $\Delta_q=\frac{1}{q}+\frac{1}{|L_2(q)|}<\frac{2}{q}$ if $u=u_I$, and
$\Delta_q=\frac{2}{q+1}$ if $u=u_\II$.  So 
$0<\Delta_q<\epsilon$. 
The sequence $u(\chi_q^n)$ ($n=1,2,\ldots$) therefore starts out greater than $1-\epsilon$ and 
tends monotonically to $0$ with steps of size less than $\epsilon$. Hence, for any $L\in[0,1]$, there exists an $n$ such that 
$|u(\chi_q^n)-L|<\epsilon$.

For $\theta_\II$, choose an integer $n\geq 2$ such that $2^{n-1}>\frac{3}{\epsilon}$.
Let $\gamma$ be the irreducible character $\eta_1$ of the dihedral group $G_n$ given by \eqref{Dihedral Characters} with $h=1$.
By Theorem~\ref{cav}, $u_\II(\gamma)=0$, so $\theta_\II(\gamma^k)=z_\II(\gamma^k)$ for each positive integer $k$.
By Lemma~\ref{Dihedral Vanishing Lemma}\eqref{D class}, 
$\gamma$ 
vanishes on exactly $3$ classes, so $z_\II(\gamma)=\frac{3}{2^{n-1}+3}<\epsilon$. 
Therefore, by Lemma~\ref{Limit Lemma}, the sequence $\theta_\II(\gamma^k)$ ($k=1,2,\ldots$) starts out less than $\epsilon$ and tends to $1$ with steps of size less than $\epsilon$. Hence, for any $L\in [0,1]$, there exists a $k$ such that 
$|\theta_\II(\gamma^k)-L|<\epsilon$.
\end{proof}

\begin{proof}[Proof of Theorem~\ref{global 01}]%
Let $\epsilon>0$. Consider first $\theta_\II$. By Lemma~\ref{Dihedral Lemma}, the nilpotent dihedral groups $G_n$ satisfy
\[u_\II(G_n)=\frac{4}{2^{n-1}+3},\qquad
  z_\II(G_n)=\frac{2^{n-2}(n+3)-2}{(2^{n-1}+3)^2}.\]
Let $l\geq 2$ be an integer such that $u_\II(G_l)<\frac{\epsilon}{2}$ and $z_\II(G_l)<\frac{\epsilon}{2}$.
By Lemma~\ref{Prop Nilp}, $u_\II(G_l^k)=u_\II(G_l)^k$ for each positive integer $k$,
 so the sequence $u_\II(G_l^k)$ (${k=1,2,\ldots}$) starts out less than $\frac{\epsilon}{2}$ and tends monotonically to $0$.
By Lemma~\ref{Limit Lemma},
the sequence $z_\II(G_l^k)$ ($k=1,2,\ldots$) starts out less than $\frac{\epsilon}{2}$ and tends to $1$ with steps of size
less than $\frac{\epsilon}{2}$.
The sequence $\theta_\II(G_l^k)$ ($k=1,2,\ldots$) therefore starts out
less than $\epsilon$ and tends to~$1$ with steps of size less than $\epsilon$. Hence, for any $L\in[0,1]$, there exists a $k$ such that ${|\theta_\II(G_l^k)-L|<\epsilon}$.

Let $u\in\{u_I,u_\II\}$.
By Lemma~\ref{Extraspecial Lemma}, we have that the extraspecial groups $H_n$ of order $2^{2n+1}$ satisfy $u(H_n)=\frac{2^{2n}}{2^{2n}+1 }$.
Let $m$ be a positive integer such that
${|u(H_m)-1|<\epsilon}$.
By  Lemma~\ref{Prop Nilp}, we have 
$u(H_m^k)=u(H_m)^k$ for each positive integer $k$. 
Therefore, the sequence $u(H_m^k)$ ($k=1,2,\ldots$) starts out less than $\epsilon$ away from $1$ and tends to $0$ 
with steps of size less than $\epsilon$. Hence, for any $L\in [0,1]$, there exists a $k$ such that 
$|u(H_m^k)-L|<\epsilon$.

Let $z\in\{z_I,z_\II\}$. Let $r$ be a positive integer such that $z(H_r)<\epsilon$, which is possible by Lemma~\ref{Extraspecial Lemma}.
By Lemma~\ref{Limit Lemma}, the sequence $z(H_r^k)$ ($k=1,2,\ldots$)
starts out less than $\epsilon$ and tends to $1$ with steps of size less than $\epsilon$. Hence, for any $L\in[0,1]$, there exists a $k$ such that
$|z(H_r^k)-L|<\epsilon$. (Cf. \cite{Miller2019}.)
\end{proof}

\end{document}